\documentclass[
12pt, 
a4paper, 
oneside, 
headinclude,footinclude, 
]{article}

\usepackage[
nochapters, 
beramono, 
eulermath,
pdfspacing, 
dottedtoc 
]{classicthesis} 

\usepackage{arsclassica} 

\usepackage[T1]{fontenc} 

\usepackage[utf8]{inputenc} 

\usepackage{graphicx} 

\usepackage{enumitem} 

\usepackage{amsmath,amssymb,amsthm} 

\usepackage{varioref} 

\hypersetup{
colorlinks=true, breaklinks=true, bookmarks=true,bookmarksnumbered,
urlcolor=webbrown, linkcolor=RoyalBlue, citecolor=webgreen, 
pdftitle={}, 
pdfauthor={\textcopyright}, 
pdfsubject={}, 
pdfkeywords={}, 
pdfcreator={pdfLaTeX}, 
pdfproducer={LaTeX with hyperref and ClassicThesis} 
} 

\usepackage{tikz}
\usetikzlibrary{decorations.pathreplacing, calligraphy}

\usepackage{hyperref}
\usepackage{bm, stmaryrd}											
\usepackage{a4wide}
\usepackage{mathdots}
\usepackage[english]{babel}
\usepackage{xcolor}
\usepackage{float}
\usepackage{bigints}
\usepackage{subcaption}
\usepackage{accents}

\allowdisplaybreaks

\newtheorem{thm}{Theorem}[section]

\newtheorem{prop}[thm]{Proposition}
\newtheorem{conj}[thm]{Conjecture}

\newtheorem{defi}[thm]{Definition}

\newcommand{\osc}{\detokenize{osc}}

\theoremstyle{definition}				

\numberwithin{equation}{section}

\newcommand{\R}{\mathbb{R}}

\newcommand{\N}{\mathbb{N}}
\newcommand{\Z}{\mathbb{Z}}

\newcommand{\dist}{\detokenize{dist}}

\title{\normalfont\spacedallcaps{Rational Points near Monofractal Curves and the Strong Oscillation Principle}} 

\author{\spacedlowsmallcaps{Faustin ADICEAM\textsuperscript{1} \& \spacedlowsmallcaps{Volodymyr PAVLENKOV\textsuperscript{2}} \& Evgeniy ZORIN\textsuperscript{3} } }

\date{} 

\begin{document}

\maketitle


\renewcommand{\sectionmark}[1]{\markright{\spacedlowsmallcaps{#1}}} 
\lehead{\mbox{\llap{\small\thepage\kern1em\color{halfgray} \vline}\color{halfgray}\hspace{0.5em}\rightmark\hfil}} 

\pagestyle{scrheadings} 

\begin{flushright}
\textit{Florian Luca (1969 -- 2026)\\ in memoriam\\}
\end{flushright} 


\begin{abstract} 
\noindent In their previous work devoted to the distribution of rational points near Brownian motion, the authors conjectured the existence of an \emph{oscillation principle} governing the asymptotic behavior of the number of rational points with bounded denomi\-nators near the graph of a monofractal curve. In this note, a weaker form of this conjecture is shown to hold for a broad class of deterministic fractal curves. These include the classical Takagi and Weierstrass nowhere differentiable functions, and indeed   a prevalent (i.e.~"large") class of functions among those which are  H\"older continuous with a given exponent of regularity. This constitutes the first instance of deterministic fractal curves for which a precise count of the rational points under consideration is established.
\end{abstract}


\let\thefootnote\relax\footnotetext{\textsuperscript{1} {Laboratoire d’analyse et de mathématiques appliquées (LAMA), Université Paris-Est Créteil, France,} \texttt{faustin.adiceam@u-pec.fr}}
\let\thefootnote\relax\footnotetext{\textsuperscript{2}Department of Mathematics, University of York,
Heslington, York, YO10 5DD, England, {\texttt{volodymr.pavlenkov@york.ac.uk}}.}
\let\thefootnote\relax\footnotetext{\textsuperscript{3} Department of Mathematics, University of York,
Heslington, York, YO10 5DD, England, \texttt{evgeniy.zorin@york.ac.uk}.\\}

\section{Introduction}

\noindent Fix throughout  a real-valued function  $f$ defined over a compact interval $I$ with nonempty interior. Given an integer $Q\ge 1$ and reals $\eta, \delta>0$, denote by 
\begin{equation}\label{tubnegh}
\mathfrak{I}_f(\eta)\;=\; \left\{\left(s,t\right)\in\R^2\; :\; \inf_{x\in I} \max\left\{\left| x-s\right|, \left|f(x)-t\right|\right\}\le \eta\right\}
\end{equation}
its $\eta$--tubular neighburhood (with respect to the sup norm) and by 
\begin{equation*}
\mathcal{N}_f(\delta, Q)\;=\; \#\left\{\left(p,q,r\right)\in\Z^3\; :\; 1\le q\le Q\quad\textrm{and}\quad \left(\frac{p}{q}, \frac{r}{q}\right)\in \mathfrak{I}_f\!\left(\frac{\delta}{Q}\right)\right\}
\end{equation*}
the number of rational points $\left(p/q, r/q\right)$ with denominator $q$ bounded by $Q$  (counted with multiplicity)  lying in the $(\delta/Q)$--neighbourhood of the curve defined by the map~$f$. Let also $\mathcal{A}_f(\eta)$ denote the area of the domain $\mathfrak{I}_f(\eta)$, whenever measurable.\\

\noindent The Area Heuristic predicts the existence of a (potentially infinite) exponent $\alpha>0$ for which 
\begin{equation*}
\mathcal{N}_f(\delta, Q)\;\asymp\; \mathcal{A}_f\!\left(\frac{\delta}{Q}\right)\cdot Q^3
\end{equation*}
as long as $\delta\gg Q^{-\alpha}$ (\textsuperscript{4})\let\thefootnote\relax\footnotetext{\textsuperscript{4}When $g,h$ are two real--valued, positive functions depending on a parameter $x$, the notations $f(x)\ll g(x)$ and $f(x)\gg g(x)$ mean the existence of constants $c, c'>0$ independent of the variable $x$  such that $f(x)\le c\cdot g(x)$ and $f(x)\ge c'\cdot g(x)$ for all $x$, respectively. The notation $f(x)\asymp g(x)$ means that the two relations $f(x)\ll g(x)$  and $f(x)\gg g(x)$ are simultaneously met.}. The justification of this heuristic and the state of the art regar\-ding its validity are thoroughly outlined in~\cite[\S1.1]{APZ}. It suffices to mention here that it is rather well--understood in the case of sufficiently regular curves. For non-differentiable ones however, the only known case for which the Area Heuristic has been established is the case of an almost sure realisation of Brownian motion. This has been achieved in~\cite[\S1.3]{APZ}. \\

\noindent To develop the theory   of the distribution of  rational points near irregular curves,  the authors conjecture in~\cite[\S1.4]{APZ}  the existence of a so--called \emph{Oscillation Principle} governing the asymptotic behavior of the counting function under consideration in the case of \emph{monofractal} curves. In order to state this principle, recall that the local H\"older exponent of the function $f$ at a point $x\in I$ is defined as
\begin{equation}\label{holdr}
h_f(x)\;=\; \liminf_{\eta\rightarrow 0^+}\; \frac{\log \omega_f(x,\eta)}{\log \eta},
\end{equation}
where 
\begin{equation}\label{defomega} 
 \omega_f(x,\eta)\;=\; \sup_{\stackrel{y\in I}{0<\left|x-y\right|\le \eta}} \left|f(x)-f(y)\right|
\end{equation}
is the oscillation of $f$ in a $\eta$--neighbourhood of $x\in I$ (conventionally, $h_f(x)=\infty$ if $f$ is locally constant around $x$). The map $f$ is then \emph{monofractal} if there exists $\gamma\in [0,1]$ such that 
\begin{equation}\label{monoexpo}
h_f(x)\;=\;\gamma \qquad \textrm{for all}\qquad x\in I.
\end{equation}

\noindent Its \emph{average oscillation} of order $k\ge 1$ over a subinterval  $J\subset I$  is the quantity
\begin{equation}\label{sfJk}
S_f\!\left(J,k\right)\;=\; \frac{1}{2^k}\sum_{j=1}^{2^k} \osc\!\left(f, J^{(j)}_{k}\right).
\end{equation}
Here, $J^{(1)}_{k}, \dots, J^{\left(2^k\right)}_{k}$ denote the partition of the interval $J$  into $2^k$ closed subintervals of equal length (up to boundary points). Furthermore, given a subinterval $\widetilde{J}\subset I$, one has set 
\begin{equation}\label{defosc}
\osc\!\left(f, \widetilde{J}\right)\;=\; \sup_{x,y\in \widetilde{J}} \left|f(x)-f(y)\right|.
\end{equation}
Let finally 
\begin{equation}\label{defdi}
\mathfrak{D}(I) 
\end{equation}
be the collection of all those intervals obtained by repeated bisections of the interval $I$, with the requirement that $I \in \mathfrak{D}(I)$. The function $f$ is \emph{uniformly oscillating} with exponent $\beta \in (0, 1]$ if, for every $J \in \mathfrak{D}(I)$, 
\begin{equation}\label{asympsfJk}
S_f(J,k)\;\asymp\; \left(\frac{\left|J\right|}{2^k}\right)^{\beta} \qquad \textrm{as }  k \rightarrow \infty,
\end{equation}
where,  here and throughout, $\left|J\right|$ denotes the length of the interval $J$,  and where the implicit constant may depend on $f$ and $J$, but not on $k$.  \\

\noindent As proved in~\cite[\S1.4]{APZ}, condition~\eqref{asympsfJk} is enough to guarantee that 
\begin{equation*}
\mathcal{A}_f(\eta)\;\asymp\; \eta^{\beta}.
\end{equation*}
This motivates the following statement claiming 
that the Area Heuristic is satisfied by a large class of monofractal curves~: 

\begin{conj}[Oscillation Principle, \cite{APZ}]\label{conjoscprin}
Assume that the map $f: I\rightarrow \R$  is continuous,  monofractal with exponent $\gamma\in (0, 1)$, and that it is also   uniformly oscillating. Then, the Area Heuristic holds in the sense that there exists an exponent $\alpha>0$ such that $$\mathcal{N}_f(\delta, Q)\;\asymp\; \left(\frac{\delta}{Q}\right)^\gamma\cdot Q^3\qquad \textrm{whenever }\qquad \delta\gg Q^{-\alpha}.$$
\end{conj}

\noindent To be specific, there is a subtle difference between the above statement and the one appearing as Conjecture~1.2 in~\cite[\S1.4]{APZ}; namely, it is not assumed in the above that the monofractal exponent $\gamma$ of the map $f$ equals its uniformly oscillating exponent $\beta$. As claimed without a proof in a footnote in~\cite[\S1.4]{APZ} , this is because these two quantities necessarily coincide~: 

\begin{prop}[Monofractality and Uniform Oscillation]\label{monoandUO}
Assume that the continuous map $f: I\rightarrow \R$ is monofractal with exponent $\gamma\in  [0,1]$ and that it admits  a uniformly oscillating exponent $\beta\in  [0,1]$. Then, these two exponents coincide; that is, $\beta=\gamma$. 
\end{prop}

\noindent The main purpose of this note is to show that  the Oscillation Principle (Conjecture~\ref{conjoscprin}) holds when its assumptions are strengthened according to this definition~: 

\begin{defi}[Strong Uniform Oscillation]\label{suo} 
The   function  $f: I\rightarrow \R$ is \emph{strongly uniformly oscillating} with exponent $\beta\in (0,1)$ if there exist constants $c_f, C_f, \eta_f>0$ depen\-ding only on  $f$ such that for every subinterval $J\subset I$, 
\begin{equation*}
c_f\cdot \left|J\right|^\beta \;\le\; \osc\!\left(f, J\right)\;\le \; C_f\cdot \left|J\right|^\beta \qquad \textrm{whenever}\qquad 0<\left|J\right|\le \eta_f.
\end{equation*}
Here, $\osc\!\left(f, J\right)$ is the quantity  defined in~\eqref{defosc}. 
\end{defi} 

\noindent In fact, for the considered range of exponents $\beta\in(0,1)$, the Strong Uniform Oscillation
property is equivalent to the concept of strong mono--H\"older regularity appearing in the Multifractal Analysis literature --- see,
e.g.,~\cite{CN}. \\

\noindent Clearly, a strongly uniformly oscillating map is, in particular,  continuous and nowhere differentiable. As shown in the following statement, this concept is a genuine strengthe\-ning of that   uniform oscillation~:  

\begin{prop}[Strong Uniform Oscillation implies  Uniform Oscillation and Monofractality]\label{SUO=>UOandM}
If the  map $f: I\rightarrow \R$ is strongly uniformly oscillating with exponent $\beta\in (0,1)$, then it is both  uniformly oscillating and monofractal  with the same exponent $\beta$. 
\end{prop}

\noindent That being said, the concept of Strong Uniform Oscillation remains strictly stronger than that of Uniform Oscillation as introduced in~\eqref{asympsfJk}.  An explicit example of a uniformly but not strongly uniformly oscillating map is an almost sure realisation $\mathcal{B}=\left(B_t\right)_{t\ge 0}$ of Brownian motion~: its uniform oscillation property is  indeed established in~\cite[\S1.4]{APZ}. However,  since from Levy's  Modulus of Continuity Theorem~\cite[27, Theorem~1.14 \& Remark~1.18]{MP}, 
one has that almost surely,
\begin{equation*}
\sup_{\stackrel{J\subset I}{0<\left|J\right|\le h}}\osc\!\left(\mathcal{B}, J\right)\;\asymp\; \sqrt{h\cdot\log\left(\frac{1}{h}\right)}\qquad \textrm{as }\qquad h\rightarrow 0^+, 
\end{equation*}
it is an easy exercise to check that  the conditions of Definition~\ref{suo} are not met. \\

\noindent The main result of this note reads as follows~: 

\begin{thm}[Strong Oscillation Principle]\label{mainthm}
Let  $f: I\rightarrow \R$ be a strongly uniformly oscilla\-ting function with exponent $\gamma\in (0,1)$. Fix any real 
\begin{equation}\label{rangealpha}
0\;<\;\alpha\;<\;\frac{1-\gamma}{\gamma},
\end{equation}
and let $\delta: \N\rightarrow (0, 1/4]$ be any map meeting the lower bound
\begin{equation}\label{lowboundalp}
\delta\!\left(Q\right)\;\gg\; Q^{-\alpha}.
\end{equation}
Then, the Area Heuristic holds in the sense that,   as $Q$ tends to infinity, $$\mathcal{N}_f(\delta(Q), Q)\;\asymp\; \left(\frac{\delta(Q)}{Q}\right)^\gamma\cdot Q^3.$$
\end{thm}

\noindent This theorem enables one to confirm the validity of  the Oscillation Principle (Conjecture~\ref{conjoscprin}) for a large class of monofractal curves. That being said, the example of an almost sure realisation of Brownian motion, which meets both the assumptions of the conjecture (as noted above) and its conclusion from the Second Main Theorem in~\cite{APZ}, shows that establishing the full conjecture has the potential  to capture a ge\-nui\-nely larger class of monofractal curves. \\

\noindent The proof that an almost sure realisation of Brownian motion falls within the scope of the Oscillation Principle is achieved in~\cite{APZ} thanks to the second moment method in Probability Theory. It is nevertheless left  as an open problem therein to establish the conjecture in the case of a \emph{deterministic} monofractal curve. Theorem~\ref{mainthm} enables one to precisely achieve this goal for the following families of fractal curves defined with the help of reals 
\begin{equation*}
0\;<\;\frac{1}{b}\;<\;a\;<\;1
\end{equation*}
and of an auxiliary sequence of real numbers $\bm{\varphi}=(\varphi_m)_{m\ge 0}$~:

\begin{enumerate}
\item the \emph{Inhomogeneous Takagi–van der Waerden Function} 
$$T_{a,b}^{\bm{\varphi}}~: x\in\R\;\mapsto\; \sum_{m=0}^{\infty} a^m\cdot\left\|b^mx+\varphi_m\right\|.$$ Here, one has set $\left\|t\right\|=\textrm{dist}(t,\Z)$ when  $t\in\R$. It is indeed established in~\cite[Theorem~9.2.1 \& Remark~8.5.2(e)]{JP} that $T_{a,b}^{\bm{\varphi}}$ meets the strong uniform oscillation property with the exponent
\begin{equation}\label{gammaabgen}
\gamma(a,b)
\;=\;
-\frac{\log a}{\log b}
\;\in\;(0,1).
\end{equation}

\item the \emph{Inhomogeneous Weierstrass Function} 
\begin{equation*}
W_{(a,b)}^{\bm{\varphi}}\;:\; x\in I\;\mapsto\;  \sum_{m=0}^{\infty}a^m\cos\!\left(2\pi\cdot\left( b^m x+\varphi_m\right)\right).
\end{equation*} 
It is indeed easily deduced from~\cite[ Theorem 8.3.1]{JP} that $W_{(a,b)}^{\bm{\varphi}}$ meets the strong uniform oscillation property with the same exponent as in~\eqref{gammaabgen}.

\item the \emph{Generalised Inhomogeneous Weierstrass Function} 
defined with the help of a non-constant, $1$--periodic Lipschitz continuous function $g:\R\rightarrow\R$ as 
\begin{equation*}
W_{(a,b)}^{\bm{\varphi}, g}\;:\; x\in I\;\mapsto\;  \sum_{m=0}^{\infty} a^m g\!\left(b^m x+\varphi_m\right).
\end{equation*} 
The situation is here more subtle and incomplete. Indeed, while the upper bound $$\osc\!\left(W_{g,(a,b)}^{\bm{\varphi}},J\right)
\;\ll\;
|J|^{\gamma(a,b)}$$ is readily established for  every sufficiently small interval $J\subset I$ from the Lipschitz continuity of the map $g$, the matching lower bound is only known to hold  conditionally in two cases~: (a) when $b>1$ is transcendental --- this is due to Heurteaux~\cite{He} and (b) when the inhomogeneous phase $\bm{\varphi}$  vanishes identically whereas the auxiliary map $x\mapsto \sum_{m\in\Z}a^m \cdot\left(g(b^mx)-g(0)\right)$ is nonzero --- this is due to Hu and Lau~\cite[Theorem~4.1]{HL}. 
\end{enumerate}

\noindent This list is, of course, not meant to be exhaustive. For instance, the so-called class of Kiesswetter-type functions defined and analysed in~\cite[Remark~5.3.12.]{JP} provides another family of strongly uniformly oscillating maps. \\

\noindent From the point of view of Multifractal Analysis, Proposition~3 in~\cite{CN} furnishes a sufficient wavelet criterion for a map to be strongly uniformly oscillating. This criterion is not tied to any particular series representation and is especially useful to construct a large class of functions the Strong Oscillation Principle (Theorem~\ref{mainthm}) applies to. \\

\noindent Furthermore, Theorem~1 in~\cite{CN} shows that among the functions which are H\"older continuous with some exponent $\gamma\in (0,1)$, those which have the strong uniform oscillation property are prevalent (i.e.~"large" in a suitable sense --- see the definition of this concept \emph{ibid.}). Functions meeting the Strong Oscillation Principle (Theorem~\ref{mainthm}) are thus also prevalent in this class of H\"older continuous functions.

\paragraph{Organisation of the paper.} The properties of a (strongly) uniformly oscillating  map stated in Propositions~\ref{monoandUO} and~\ref{SUO=>UOandM} are established in Section~\ref{sec2}. Theorem~\ref{mainthm} concerned with the Strong Oscillation Principle is proved in Section~\ref{sec3}. 

\paragraph{Acknowledgments.} The first--named author's   work was supported by the French \emph{Agence Nationale de la Recherche} through grant ANR-25-CE40-1961-01. The other two co-authors were supported by EPSRC grant UKRI 2768. \color{black}

\section{Properties of (Strongly) Uniformly Oscillating Functions}\label{sec2}

\subsection{Monofractality and Uniform Oscillation}

\noindent This subsection is devoted to the proof of Proposition~\ref{monoandUO}. To this end, recall that $\beta\in [0,1]$ denotes the uniformly oscillating exponent of the continuous map $f:I\rightarrow \R$ (as defined in~\eqref{asympsfJk}) and $\gamma \in  [0,1]$ its monofractal exponent (as defined in~\eqref{monoexpo}). The proof is achieved in two steps.

\begin{proof}[Proof of the inequality $\gamma\le \beta$]
From the assumption of monofractality and from the de\-fi\-ni\-tion of the oscillation quantity $ \omega_f(x,\eta)$  in~\eqref{defomega}, for any $\varepsilon>0$ and any $x\in I$, there exists $\eta_0=\eta_0(x, \varepsilon)>0$ such that for all $\eta\in\left(0, \eta_0\right)$, one has $ \omega_f(x,\eta)\le \eta^{\gamma-\varepsilon}$. As a consequence, the interval $I$ can be decomposed as $$I\;=\; \bigcup_{N=1}^\infty F_N(\gamma, \varepsilon),$$ where given an integer $N\ge 1$, $$F_N(\gamma, \varepsilon)\;=\; \left\{x\in I\;:\; \omega_f(x,\eta)\le \eta^{\gamma-\varepsilon}\quad \textrm{for all}\quad 0<\eta\le \frac{1}{N}\right\}.$$

\noindent The continuity of the function $f$ implies that each set $F_N(\gamma, \varepsilon)$ is closed. It then follows from the Baire Category Theorem that there exists an index $N_0\ge 1$ such that $F_{N_0}(\gamma, \varepsilon)$ has nonempty interior. As a consequence, there exists a dyadic interval $J\in\mathcal{D}(I)$ contained in $F_{N_0}(\gamma, \varepsilon)$ (where the collection of intervals $\mathcal{D}(I)$ is defined in~\eqref{defdi}) and some parameter $\eta'>0$ such that for all $x\in J$ and all $\eta\in (0, \eta')$, one has that $$ \omega_f(x,\eta)\;\le\; \eta^{\gamma-\varepsilon}.$$

\noindent Fix an integer $k\ge 1$. Adopting the notation of the introduction, denote by $J^{(1)}_{k}, \dots, J^{\left(2^k\right)}_{k}$  the partition of the interval $J$  into $2^k$ closed subintervals of equal length $$\eta_k\; :=\; \frac{\left|J\right|}{2^k}$$ 
(up to boundary points). Since $J\subset F_{N_0}(\gamma, \varepsilon)$, given $x,y\in J^{(\ell)}_{k}$ (where $1\le \ell\le 2^k$), one has  that $$\left|f(x)-f(y)\right|\;\underset{\eqref{defomega}}{\le}\;  \omega_f(x,\eta_k)\;\le\; \eta_k^{\gamma-\varepsilon}.$$ In particular, $\osc\!\left(f, J^{(\ell)}_{k}\right)\le \eta_k^{\gamma-\varepsilon}$ for every $1\le \ell\le 2^k$ in such a way that 
\begin{equation*}
S_f\!\left(J,k\right)\;\underset{\eqref{sfJk}}{\le}\;\eta_k^{\gamma-\varepsilon}.
\end{equation*}
Also, the assumption of uniform oscillation yields the existence of a constant $c_J>0$ such that for all $k\ge 1$, $$S_f\!\left(J,k\right)\;\ge\; c_J\cdot \eta_k^\beta.$$ Thus, $ c_J\cdot \eta_k^\beta\le \eta_k^{\gamma-\varepsilon}$. Since, clearly, $\eta_k\rightarrow 0$ as $k\rightarrow \infty$, one obtains that $\beta\ge \gamma-\varepsilon$. Upon letting $\varepsilon$ tend to zero, this yields the sought inequality $\beta\ge \gamma.$
\end{proof}

\vspace{2mm}

\begin{proof}[Proof of the inequality $\gamma\ge \beta$] Assume for a contradiction that 
$\beta> \gamma$
and, given an integer $k\ge 1$, set in this proof  
\begin{equation}\label{deetak}
\eta_k\;=\; \frac{\left|I\right|}{2^k}\cdotp  
\end{equation}
Since by assumption, $h_f(x)=\gamma$ for all $x\in I$, there exist infinitely many indices $k\ge 1$ such that $$\omega_f\!\left(x, \eta_k\right)\;\ge\; \eta_k^{\gamma+\varepsilon}$$ for any given small enough value of $\varepsilon\in \left(0, \beta-\gamma\right)$. (To see this,  note that a sequence $\left(\delta_i\right)_{i\ge 1}$ tending to 0 such that $\omega_f\!\left(x, \delta_i\right)\;\ge\; \delta_i^{\gamma+\varepsilon}$ for all $i\ge 1$ exists by the definition of the quantity $h_f(x)=\gamma$; it then suffices to sandwich each $\delta_i$ between two consecutive dyadic scales determined by the sequence $\left(\eta_k\right)_{k\ge 1}$, to rely on the monotonicity of the map $\eta\mapsto \omega_f\!\left(x, \eta\right)$ and to adjust if needed the value of $\varepsilon>0$.)\\

\noindent Let then 
\begin{equation*}
E_k\!\left(\gamma, \varepsilon\right)\;=\; \left\{x\in I\;:\; \omega_f\!\left(x, \eta_k\right)\;\ge\; \eta_k^{\gamma+\varepsilon}\right\}.
\end{equation*}
 The above discussion shows that 
 \begin{equation}\label{limsupe_kI}
 \limsup_{k\rightarrow\infty} E_k\!\left(\gamma, \varepsilon\right)\;=\; I.
 \end{equation} 
 Let  $I^{(1)}_{k}, \dots, I^{\left(2^k\right)}_{k}$ denote the partition (up to boundary points) of the interval $I$  into $2^k$ closed subintervals of equal length $\eta_k$ (as defined in~\eqref{deetak}). Set for the sake of simplicity of notation 
 \begin{equation*}
 \sigma_k(\ell)\;=\; \osc\!\left(f, I^{(\ell)}_{k}\right)
 \end{equation*}
when $1\le \ell\le 2^k$. If $x\in I^{(\ell)}_{k}$ and $\left|x-y\right|\le \eta_k$, then it should be clear that $y$ belongs to one of the three intervals $I^{(\ell-1)}_{k}$, $I^{(\ell)}_{k}$ or $I^{(\ell+1)}_{k}$, whenever the indices are well--defined. As a consequence, by the Triangle Inequality, 
\begin{equation*}
\omega_f\!\left(x, \eta_k\right)\;\le\;   \sigma_k(\ell-1)\;+\;  \sigma_k(\ell)\;+\;  \sigma_k(\ell+1),
\end{equation*}
where, conventionally,  the quantities vanish when the corresponding indices are not well--defined. This shows that when $x\in I^{(\ell)}_{k}\cap E_k\!\left(\gamma, \varepsilon\right)$, one has that $$\max_{\ell-1\le p\le \ell+1}\left\{\sigma_k(p)\right\}\;\ge\; \frac{\eta_k^{\gamma+\varepsilon}}{3}\cdotp $$

\noindent Upon setting $$\mathfrak{B}_k\;=\; \left\{1\le \ell\le 2^k\; : \;  \sigma_k(\ell)\;\ge\; \frac{\eta_k^{\gamma+\varepsilon}}{3}\right\},$$ one thus obtains that the set $E_k\!\left(\gamma, \varepsilon\right)$ is contained in the union of the intervals indexed by $\mathfrak{B}_k$ and their immediate neighbours, whence the bound
\begin{equation}\label{blaaa}
\left|E_k\right|\;\le\; 3\cdot \eta_k\cdot \#\mathfrak{B}_k.
\end{equation}
Since, clearly, 
\begin{equation*}
\frac{1}{3}\cdot \eta_k^{\gamma+\varepsilon}\cdot \#\mathfrak{B}_k\;\le\; \sum_{\ell=1}^{2^k}  \sigma_k(\ell)\;=\; 2^k\cdot S_f\!\left(I, k\right),
\end{equation*}
it follows that $$ \#\mathfrak{B}_k\;\le\; 3\cdot \eta_k^{-\left(\gamma+\varepsilon\right)}\cdot 2^k\cdot S_f\!\left(I, k\right),$$ and therefore that $$\left|E_k\right|\;\underset{\eqref{blaaa}}{\le}\; 9\cdot \eta_k^{1-\gamma-\varepsilon}\cdot 2^k\cdot S_f\!\left(I, k\right)\;=\; 9\cdot\left|I\right|\cdot \eta_k^{-\left(\gamma+\varepsilon\right)}\cdot S_f\!\left(I, k\right),$$ where the last equation uses the relation 
\begin{equation}\label{relki}
2^k\cdot\eta_k=\left|I\right|.
\end{equation}

\noindent The uniform oscillation assumption implies that $S_f\!\left(I, k\right)\ll \eta_k^{\beta}$  and yields that $\left|E_k\right|\ll \eta_k^{\beta-\gamma-\varepsilon}$. As a consequence, taking into account identity~\eqref{relki} and the inequality $\beta-\gamma-\varepsilon>0$, $$\sum_{k=1}^{\infty}\left|E_k\right|\;\ll\; \sum_{k=1}^{\infty}2^{-k\left(\beta-\gamma-\varepsilon\right)}\;<\; \infty.$$ From the Borel--Cantelli lemma, $\limsup_{k\rightarrow\infty}E_k$ happens to be null set, which contradicts~\eqref{limsupe_kI} (recall that the interval $I$ is throughout assumed to have nonempty interior). This concludes the proof of the sought complementary inequality $\gamma\ge \beta$. 
\end{proof}

\noindent The proof of Proposition~\ref{monoandUO} is complete.

\subsection{Strongly Uniformly Oscillating Maps}

\noindent This subsection is devoted to the proof of Proposition~\ref{SUO=>UOandM}. 

\begin{proof}
Let $x\in I$ and $\eta>0$. From the definitions of the local H\"older exponent $h_f(x)$ in~\eqref{holdr} and of the measure of local oscillation $\omega_f(x,\eta)$ in~\eqref{defomega}, it is clear that the strong uniform oscillation property stated in Definition~\ref{suo} implies that $\omega_f(x,\eta)\ll\eta^{\gamma}$. To show the monofractality with exponent $\gamma$ of the map $f$, one needs to establish the complementary bound $\omega_f(x,\eta)\gg\eta^{\gamma}$.\\

\noindent To this end, fix an subinterval $J(\eta)\subset I\cap \left[x-\eta, x+\eta\right]$ with length $\eta$. From the lower bound in  the definition of the property of strong uniform oscillation (Definition~\ref{suo}), there exist $u,v\in J(\eta)$ such that $\left|f(u)-f(v)\right|\gg \eta^\gamma$. As a consequence,  from the Triangle inequality, $$\omega_f(x,\eta)\;\ge\; \max\left\{\left|f(v)-f(x)\right|, \left|f(u)-f(x)\right|\right\}\;\gg\; \eta^\gamma.$$
This shows that $\omega_f(x,\eta)\asymp \eta^\gamma$ and thus that the map $f$ is indeed monofractal with exponent $\gamma$ under the assumption of strong uniform oscillation with exponent $\gamma.$\\

\noindent To prove that it is also uniformly oscillating with exponent $\gamma$, adopt the notation of the introduction and let $J\in\mathfrak{D}(I)$. Applying the definition of strong uniform oscillation to the $2^k$ subintervals $J^{(1)}_{k}, \dots, J^{\left(2^k\right)}_{k}$ (each with length $\left|J\right|/2^k$), one immediately obtains that the quantity $S_f\!\left(J,k\right)$ defined in~\eqref{sfJk}  satisfies the relation $S_f\!\left(J,k\right)\asymp \left(\left|J\right|/2^k\right)^\gamma$, whence the property of uniform oscillation. This completes the proof of Proposition~\ref{SUO=>UOandM}.
\end{proof}

\section{Proof of the Strong Oscillation Principle}\label{sec3}

\noindent This section is devoted to the proof of the Main Theorem~\ref{mainthm}.

\begin{proof} Set in this proof 
\begin{equation}\label{defetaQ}
\eta_Q\;=\; \frac{\delta(Q)}{Q}\cdotp
\end{equation} 
Given integers $1\le q\le Q$ and $p\in\Z$, consider the vertical section 
\begin{equation*}
V_{p,q}\!\left(\eta_Q\right)\;=\; \left\{y\in\R\; :\; \left(\frac{p}{q},y\right)\in \mathfrak{I}_f\!\left(\eta_Q\right)\right\},
\end{equation*}
where the tubular neighbourhood $ \mathfrak{I}_f\!\left(\eta_Q\right)$ is defined in~\eqref{tubnegh}. For the section $V_{p,q}\!\left(\eta_Q\right)$ to be nonempty, if is clearly necessary that 
\begin{equation}\label{interI}
\dist\!\left(\frac{p}{q}, I\right)\;\le\; \eta_Q. 
\end{equation}
When this condition holds, let 
\begin{equation*}
K_{p,q}\!\left(\eta_Q\right)\;=\; I\cap\left[\frac{p}{q}-\eta_Q, \; \frac{p}{q}+\eta_Q\right]. 
\end{equation*}
Since the image of the interval $K_{p,q}\!\left(\eta_Q\right)$ by the continuous map $f$ is again an interval, it should be clear that 
\begin{equation}\label{lenV*}
\left|V_{p,q}\!\left(\eta_Q\right)\right|\; \underset{\eqref{defosc}}{=}\; \osc\!\left(f, K_{p,q}\!\left(\eta_Q\right)\right) + 2\cdot \eta_{Q}.
\end{equation}
It then follows from the property of strong uniform oscillation (Definition~\ref{suo}) that when $Q\ge 1$ is large enough, 
\begin{equation}\label{boundV}
\left|V_{p,q}\!\left(\eta_Q\right)\right|\; \ll\; \left| K_{p,q}\!\left(\eta_Q\right)\right|^\gamma + \eta_{Q} \;\ll\; \eta_Q^\gamma.
\end{equation}
Since for a given integer $q\ge 1$, the number of integers $p$ meeting condition~\eqref{interI} is $O\!\left(q+1\right)$ (where the implicit constant is allowed to depend on the length of $I$) and since, given an interval $V\subset\R$, 
\begin{equation}\label{countVq}
q\left|V\right| -1\;\le\; \#\left\{r\in\Z\;:\; \frac{r}{q}\in V\right\}\;\le\; q\left|V\right| +1,
\end{equation}
the upper bound in~\eqref{boundV} implies  that 
\begin{equation}\label{boundsVpq}
\mathcal{N}_f(\delta, Q)\;\ll\; \sum_{q=1}^Q\left(q+1\right)\cdot\left(q\cdot\eta_Q^\gamma+1\right)\;\ll\; \eta_Q^\gamma Q^3+Q^2\;\underset{\eqref{defetaQ}}{\ll}\; \left(\frac{\delta(Q)}{Q}\right)^\gamma\cdot Q^3.
\end{equation}
This last inequality is justified as soon as $Q^2=o\left(\eta_Q^\gamma Q^3\right)$ when $Q$ tends to infinity. To see that this relation holds, it suffices to note that 
\begin{equation}\label{lowbou}
Q\cdot \eta_Q^\gamma\;\underset{\eqref{defetaQ}}{=}\; \delta(Q)^\gamma\cdot Q^{1-\gamma}\;\underset{\eqref{lowboundalp}}{\gg}\;Q^{1-\gamma-\alpha\gamma},
\end{equation} 
where the last exponent is, by assumption, positive. \\

\noindent To obtain the lower bound   matching the upper bound~\eqref{boundsVpq},  
fix an interval $I_*$ with nonempty interior lying in the interior of the interval $I$. Provided that $Q\ge 1$ is large enough (as a function of the length of $I_*$), for any fix integer $q$ in the dyadic range $\left[Q/2, Q\right]$, the number of integers $p$ such that $p/q\in I_*$ is, up to a multiplicative constant depending on $I_*$, at least $q$. \\

\noindent Even if it means increasing again the value of $Q$, one can choose an interval $$J_{p,q}\;\subset\; K_{p,q}\!\left(\eta_Q\right)\;=\; I\cap \left[\frac{p}{q}-\eta_Q; \, \frac{p}{q}+\eta_Q\right] $$ with exact length $\eta_Q$. Then, from the definition of the property of strong uniform oscillation,
\begin{equation}\label{lenV*bis}
\left|V_{p,q}\!\left(\eta_Q\right)\right|\; \underset{\eqref{lenV*}}{\ge}\; \osc\!\left(f, J_{p,q}\!\left(\eta_Q\right)\right) + 2\cdot \eta_{Q}\;\gg\;\eta_Q^\gamma.
\end{equation}
Note also that the quantity $Q\cdot \eta_Q^\gamma$ tends to infinity with $Q$ from the lower bound~\eqref{lowbou}.  As a consequence, for any integer $q\in \left[Q/2, Q\right]$, the lower bound in~\eqref{countVq} yields that
\begin{equation*}
 \#\left\{r\in\Z\;:\; \frac{r}{q}\in V_{p,q}\!\left(\eta_Q\right) \right\}\; \gg\; q\cdot \eta_Q^\gamma. 
\end{equation*}
Summing over the range of admissible values of $p$ and $q$, one thus obtains that 
\begin{equation*}
\mathcal{N}_f(\delta, Q)\;\gg\; \eta_Q^{\gamma} \sum_{Q/2\le q\le Q} q^2\;\gg\; \eta_Q^\gamma \cdot Q^3\;\underset{\eqref{defetaQ}}{=}\;  \left(\frac{\delta(Q)}{Q}\right)^\gamma\cdot Q^3.
\end{equation*}
This lower bound and the corresponding upper bound~\eqref{boundsVpq} complete the proof of Theorem~\ref{mainthm}.
\end{proof}

\bibliographystyle{unsrt}
\addcontentsline{toc}{section}{References}

\end{document}